\documentclass[11pt, reqno]{amsart}

\usepackage{latexsym}
\usepackage{amssymb}
\usepackage{mathrsfs}
\usepackage{amsmath}
\usepackage{fancybox,color}
\usepackage{enumerate}
\usepackage[latin1]{inputenc}

\newcommand{\refeq}[1]{(\ref{#1})}

\def\1{\raisebox{2pt}{\rm{$\chi$}}}

\newtheorem{theorem}{Theorem}[section]
\newtheorem{corollary}[theorem]{Corollary}
\newtheorem{lemma}[theorem]{Lemma}

\newtheorem{definition}[theorem]{Definition}

\newtheorem{example}[theorem]{Example}

\newcommand{\R}{{\mathbb R}}

\newcommand{\eps}{{\varepsilon}}
\def\1{\raisebox{2pt}{\rm{$\chi$}}}

\newcommand{\Lip}{\operatorname{Lip}}
\newcommand{\abs}[1]{\left|#1\right|}

\newcommand{\Rn}{\mathbb{R}^n}

\def\vint_#1{\mathchoice%
         {\mathop{\kern 0.2em\vrule width 0.6em height 0.69678ex depth -0.58065ex
                 \kern -0.8em \intop}\nolimits_{\kern -0.4em#1}}%
         {\mathop{\kern 0.1em\vrule width 0.5em height 0.69678ex depth -0.60387ex
                 \kern -0.6em \intop}\nolimits_{#1}}%
         {\mathop{\kern 0.1em\vrule width 0.5em height 0.69678ex
             depth -0.60387ex
                 \kern -0.6em \intop}\nolimits_{#1}}%
         {\mathop{\kern 0.1em\vrule width 0.5em height 0.69678ex depth -0.60387ex
                 \kern -0.6em \intop}\nolimits_{#1}}}
\def\vintslides_#1{\mathchoice%
         {\mathop{\kern 0.1em\vrule width 0.5em height 0.697ex depth -0.581ex
                 \kern -0.6em \intop}\nolimits_{\kern -0.4em#1}}%
         {\mathop{\kern 0.1em\vrule width 0.3em height 0.697ex depth -0.604ex
                 \kern -0.4em \intop}\nolimits_{#1}}%
         {\mathop{\kern 0.1em\vrule width 0.3em height 0.697ex depth -0.604ex
                 \kern -0.4em \intop}\nolimits_{#1}}%
         {\mathop{\kern 0.1em\vrule width 0.3em height 0.697ex depth -0.604ex
                 \kern -0.4em \intop}\nolimits_{#1}}}

\newcommand{\kint}{\vint}
\newcommand{\intav}{\vint}
\newcommand{\aveint}[2]{\mathchoice%
         {\mathop{\kern 0.2em\vrule width 0.6em height 0.69678ex depth -0.58065ex
                 \kern -0.8em \intop}\nolimits_{\kern -0.45em#1}^{#2}}%
         {\mathop{\kern 0.1em\vrule width 0.5em height 0.69678ex depth -0.60387ex
                 \kern -0.6em \intop}\nolimits_{#1}^{#2}}%
         {\mathop{\kern 0.1em\vrule width 0.5em height 0.69678ex depth -0.60387ex
                 \kern -0.6em \intop}\nolimits_{#1}^{#2}}%
         {\mathop{\kern 0.1em\vrule width 0.5em height 0.69678ex depth -0.60387ex
                 \kern -0.6em \intop}\nolimits_{#1}^{#2}}}
\newcommand{\ud}{\, d}
\newcommand{\half}{{\frac{1}{2}}}

\newcommand{\ol}{\overline}
\newcommand{\Om}{\Omega}
\newcommand{\I}{\textrm{I}}
\newcommand{\II}{\textrm{II}}
\newcommand{\dist}{\operatorname{dist}}

\newcommand{\F}{\mathcal{F}}

\begin{document}

\title[On  $p$-harmonious functions]{On the existence and uniqueness of $p$-harmonious functions}

\author[Luiro]{Hannes Luiro}
\address{Department of Mathematics and Statistics, University of
Jyv\"askyl\"a, PO~Box~35, FI-40014 Jyv\"askyl\"a, Finland}
\email{hannes.s.luiro@jyu.fi}

\author[Parviainen]{Mikko Parviainen}
\address{Department of Mathematics and Statistics, University of
Jyv\"askyl\"a, PO~Box~35, FI-40014 Jyv\"askyl\"a, Finland}
\email{mikko.j.parviainen@jyu.fi}

\author[Saksman]{Eero Saksman}
\address{Department of Mathematics and Statistics, University of
Helsinki, PO~Box~68, FI-00014 Helsinki, Finland}
\email{eero.saksman@helsinki.fi}

\date{August 15, 2012}

\begin{abstract}
We give a self-contained and short proof for the existence, uniqueness and measurability of so called $p$-harmonious functions. The proofs only use elementary analytic tools.
As  a consequence, we obtain existence, uniqueness and measurability of value functions for the tug-of-war game with noise.  
\end{abstract}

\maketitle

\section{Introduction}\label{se:intro}

There has recently been increasing interest in the tug-of-war games, partly because they seem to give new insight in the theory of nonlinear partial differential equations, see for example \cite{peresssw09, peress08}.  The dynamic programming principle (DPP)
\begin{equation}
\label{eq:dpp}
u(x)=\frac{\alpha}{2}\sup_{B_\eps(x)}u+\frac{\alpha}{2}\inf_{B_\eps(x)}u+\beta\intav_{B_\eps(x)}u\ud y
\end{equation}
$\alpha+\beta=1,\ \alpha,\beta\ge 0$, plays a central role in the theory of tug-of-war with noise.  Intuitively, the above formula says that a value of the game at a point $x$ is a sum of three possible outcomes of a game round with corresponding probabilities: either a maximizer or minimizer wins the round, or a random noise is added. In Section \ref{sec:game} we give a brief description of the game.  Moreover, many numerical solvers  for nonlinear partial differential equations are based on functions satisfying a dynamic programming equation,  see \cite{oberman11}.  

The objective of this paper is first to give a self-contained and short proof for the existence, uniqueness and measurability of a solution to \eqref{eq:dpp} with given boundary values (Theorems  \ref{th:dpp} and \ref{thm:comparison}) using only elementary analytic tools. Secondly, we obtain as a corollary (Theorem \ref{th:value}) the existence  of the value function for the tug-of-war with noise. Our proof readily verifies that  the value function for the game is the unique, measurable solution  to the dynamic programming equation \eqref{eq:dpp}. 

When proving estimates for the game, it is often necessary to fix a strategy according to minimal or maximal values of a given value function, or design a strategy that utilizes the actions of the opponent. Lemma~\ref{le:st} shows that such a strategy  can be build in a measurable way and Corollary \ref{co:strategies} that values are robust with respect to the information on the  history of the game  that the players are allowed to use. In the last section, we remark that the  analogous results in the parabolic case are almost immediate. 

One technical difficulty is that solutions to \eqref{eq:dpp} can be discontinuous. Therefore, in Section \ref{sec:cont} we also consider a continuous version by  modifying the dynamic programming principle close to the boundary. 

The proofs utilize a natural iteration $u_{i+1}=Tu_i$ of the dynamic programming principle (precise definition of $T$ is given in \eqref{eq:operator}). However, the fact that $\beta>0$ is needed to show that the convergence is uniform and that the limit satisfies \eqref{eq:dpp}. 
Example \ref{ex:nonBorel} illustrates a delicate nature of measurability. Indeed, there is a bounded Borel function $u$ such that the function $x\mapsto \sup_{\ol B_\eps(x)}u(y)$ is not Borel. 

Finally, we recall the connection to the partial differential equations. If $u$ is harmonic, then it satisfies the well known mean value property
\[
\begin{split}
u(x) =  \frac{1}{\abs{B_\eps(x)}}\int_{ B_\eps(x)} u \ud y ,
\end{split}
\]
that is \eqref{eq:dpp} locally with $\alpha =0$ and
$\beta =1$. In this case, the iterative method for approximating solutions is known as method of relaxations \cite{courantfl28, doyles00}.
On the other hand, functions satisfying
\eqref{eq:dpp} with $\alpha =1$ and $\beta =0$
are called  harmonious functions in \cite{legruyer07} and
\cite{legruyera98}. As $\eps$ goes to zero, they approximate solutions
to the infinity Laplace equation. A similar approximation result was shown in \cite{manfredipr12} for  $p$-harmonic functions (\cite{heinonenkm93}) and solutions to \eqref{eq:dpp}  with the choice
\[
\alpha=\frac{p-2}{p+n}, \qquad \beta=\frac{n+2}{p+n}.
\]
For this reason, solutions to \eqref{eq:dpp} are  also called $p$-harmonious functions.
By interpreting \eqref{eq:dpp} in the asymptotic sense as in \cite{manfredipr10}, it actually characterizes the viscosity solutions to the $p$-Laplace equation when $p>1$. There is also a version utilizing medians in \cite{hartenstiner}. 

\section{Existence and uniqueness}

In this section we give an easy argument for existence and uniqueness of functions satisfying \eqref{eq:dpp} in an arbitrary bounded domain and with prescribed Borel boundary values. 

\newcommand{\Sup}{\operatorname{Sup}}
\newcommand{\Inf}{\operatorname{Inf}}

\subsection{Existence via iteration}\label{subse:existence}

Let $\eps>0$, $\Omega\subset\R^n$ be bounded domain, $\alpha, \beta>0$, $\alpha+\beta=1\,$. The $\eps$-boundary strip of $\Omega$ is defined by
\begin{equation}
\Gamma_{\eps}=\{\,x\in \Rn\setminus\Omega\,:\,\dist(x,\Omega)\leq \varepsilon\,\}\,.
\end{equation}
Let $\mathcal{F}_\varepsilon$ denote the set of  all non-negative, Borel measurable and bounded functions defined on $\Omega_{\eps}=\Gamma_{\eps}\cup\Omega$.
We intend to iterate the operator $T$ on $\mathcal{F}_\eps$, defined by
\begin{equation}
\label{eq:operator}
\begin{split}
Tu(x)=\begin{cases}
\frac{\alpha}{2}\sup_{B_\eps(x)}u+\frac{\alpha}{2}\inf_{B_\eps(x)}u+\beta\intav_{B_\eps(x)}u\,\,&\text{ if }x\in\Omega\\
u(x)\,\,&\text{ if } x\in\Gamma_{\eps}.
\end{cases}
\end{split}
\end{equation}

In order to check that $Tu\in \mathcal{F}_\eps$ for every $u\in\mathcal{F}_\eps$ we need  to verify that $Tu$ is Borel measurable. For that purpose it is enough to show that for any bounded Borel function $v$ in $\Omega_\eps$  the functions
$$
\sup_{y\in B_\eps(x)}v(y)\quad {\rm and}\quad \inf_{y\in B_\eps(x)}v(y),\quad x\in \Omega,
$$
are Borel in $\Omega$. This follows immediately by observing that
for any $\lambda\in\R$ the set 
\begin{equation}
\label{eq:sup-meas}
\begin{split}
\{ x\in\Omega\; :\; \sup_{B_\eps(x)} v>\lambda\}=\Omega\cap \Big( \bigcup_{y\in \Omega_\eps\,:\,v(y)>\lambda}B_\eps(y)\Big)
\end{split}
\end{equation} is open.

Next we prove existence for a function satisfying the dynamic programming equation,
\begin{theorem}
\label{th:dpp}
Given a bounded Borel boundary function
$F:\Gamma_\eps\to \R$, there is a bounded Borel function $u:\Omega_\eps\to\R$ that satisfies
the DPP with the boundary values $F$, i.e.\ $u=Tu$ and $u_{|\Gamma_\eps}=F.$
In fact, $u$ is the uniform limit
$$
u=\lim_{j\to\infty} u_j, \quad {\rm with}\;\; u_{j+1}=Tu_j\;\;{\rm for}\;\; j=0,1,\ldots,
$$
where the starting point of the iteration is 
$$
u_0(x)=\begin{cases} \inf_{y\in\Gamma_\eps} F&{\rm for}\;\; x\in \Omega \;\;{\rm and}\\
F(x) & {\rm for}\;\; x\in \Gamma_\eps.
\end{cases}
$$
\end{theorem}
\begin{proof}
By the very definition  $u_1\geq u_0$, which yields that $u_2=Tu_1\geq Tu_0=u_1$. Similarly, in general $u_{j+1}\geq u_j$ in $\Omega_{\eps}$ for all $j=0,1,\ldots$. Hence the sequence $u_j$ is increasing and since it bounded from the above by $\sup_{y\in \Gamma_\eps}F(y)<\infty$ we may define the bounded  Borel function $u$ as the monotone pointwise limit
$$
u(x):=\lim_{j\to\infty}u_j(x),\quad x\in\Omega_\eps.
$$

We claim that the convergence is uniform. Suppose contrary to our claim that
\begin{equation}\label{eqM1}
M:=\lim_{j\to\infty}\sup_{x\in\Omega_\eps}(u-u_j)(x)>0.
\end{equation}
Fix arbitrary $\delta >0$ and select $k\geq 1$ large enough so that
\begin{equation}\label{eqM2}
u-u_k\leq M+\delta\quad {\rm in}\;\; \Omega_\eps.
\end{equation}
By the dominated convergence theorem,  we may also assume that 
\begin{equation}\label{eqM3}
\sup_{x\in\Omega}\beta\intav_{B_\eps(x)}(u-u_k)(y)\,dy\,\,\leq\delta\,.
\end{equation}
By \eqref{eqM1} we may choose $x_0\in \Omega$ with the property  $u(x_0)-u_{k+1}(x_0)\geq M-\delta.$ Then choose $\ell>k$ large enough so that $u(x_0)-u_{\ell+1}(x_0)<\delta$, whence it follows that
\begin{equation}\label{eqM4}
u_{\ell+1}(x_0)-u_{k+1}(x_0)\geq M-2\delta.
\end{equation}

Observe then that for any set $A$ one has
$\sup_A u_\ell-\sup_{A}u_k\leq \sup_{A}(u_\ell-u_k)\,$
and the same holds if one replaces above the suprema on the lefthandside with infima.
The iterative definition and monotonicity of the sequence $u_j$ together with estimates \eqref{eqM2}--\eqref{eqM4} yield that
\begin{eqnarray*}
M-2\delta &\leq &u_{\ell+1}(x_0)-u_{k+1}(x_0)\\
&=&\frac{\alpha}{2}\sup_{B_\eps(x_0)}u_{\ell}+\frac{\alpha}{2}\inf_{B_\eps(x_0)}u_{\ell}+\beta\intav_{B_\eps(x_0)}u_\ell\\
&&\phantom{kkkkkkk}\,\,-\big(\frac{\alpha}{2}\sup_{B_\eps(x_0)}u_{k}+\frac{\alpha}{2}\inf_{B_\eps(x_0)}u_{k}+\beta\intav_{B_\eps(x_0)}u_k\big)\\
&\leq & \alpha\sup_{B_\eps(x_0)}(u_{\ell}-u_{k})+\beta\intav_{B_\eps(x_0)}(u_\ell-u_k)\\
&\leq & \alpha\sup_{B_\eps(x_0)}(u-u_{k})+\beta\intav_{B_\eps(x_0)}(u-u_k)\\
&\leq &\alpha (M+\delta)+\delta.
\end{eqnarray*}
This is a contradiction if $\delta$ is chosen small enough.

By the uniform convergence the limit $u$ obviously satisfies the DPP and it has the right boundary values by construction.
\end{proof}

\subsection{Elementary proof of uniqueness}

The uniqueness of $p-$harmonious function with given boundary data follows immediately from the following result.

\begin{theorem}
\label{thm:comparison}
Suppose that $u$ and $u'$ are two $p$-harmonious functions with boundary values $g$ and $g'$ (respectively) on $\Gamma_{\eps}$. Then
\[\sup_{x\in\Omega}|u'-u|(x)\leq \sup_{x\in\Gamma_{\eps}}|g'-g|(x)\,.\]
\end{theorem}
\begin{proof}
It is enough to show that 
\[M:=\sup_{x\in\Omega}(u'-u)(x)\leq \sup_{x\in\Gamma_{\eps}}(g'-g)(x)=:m\,,\]
because the rest follows by symmetric argument.
Assume that the claim is not true, so that $M>m\,$.
Hence, since both $u$ and $u'$ satisfy the DPP, it follows for every $x\in\Omega$ that
\begin{align}\label{DPPseuraus}
u'(x)-u(x)=&\frac{\alpha}{2}\big(\sup_{B_\eps(x)}u'-\sup_{B_\eps(x)}u\big)\,
+\frac{\alpha}{2}\big(\inf_{B_\eps(x)}u'-\inf_{B_\eps(x)}u\big)\notag\\
&\,\,\,\,\,+\,\beta\intav_{B_\eps(x)}(u'-u)(y)\,dy\,\notag\\
\leq& \,\alpha\sup_{y\in B_\eps(x)}(u'-u)(y)\,+\beta\intav_{B_\eps(x)}(u'-u)(y)\,dy\,\notag\\
\leq&\,\alpha M\,+\,\beta\intav_{B_\eps(x)}(u'-u)(y)\,dy\,.
\end{align}

We consider the set  
\begin{equation}\label{supset}
G:=\{x\in\Omega_{\eps}\,:\,u'(x)-u(x)=M\},
\end{equation}
for which $G\subset\Omega$ by the counter assumption. Let us first show that $G$ is non-empty. For that, choose (using the boundedness of $\Omega$) a sequence $x_k\in\Omega$ such that $(u'-u)(x_k)\to M$ as $k\to\infty$ and $x_k\to x_0\in\overline{\Omega}$. Then it follows that
\begin{equation*}
\intav_{B_\eps(x_0)}(u'-u)(y)\,dy\,=\lim_{k\to\infty}\intav_{B_\eps(x_k)}(u'-u)(y)\,dy\,=\,M\,,
\end{equation*}
where the convergence follows from the absolute continuity of the integral and the last equality is a direct consequence of (\ref{DPPseuraus}). Since $u'-u\leq M$ in $B_\eps(x_0)\subset \Omega_{\varepsilon}$ we get that $u'-u=M$ almost everywhere in $B_\eps(x_0)$ and therefore $|B_\eps(x_0)\setminus G|=0\,$. This of course also implies that $G\not=\emptyset\,$. The crucial point now is that $G$ also has a following property:
\begin{equation}\label{cruc1}
\text{If}\quad x\in G,\quad\text{then} \quad \abs{B_\eps(x)\setminus G}=0.
\end{equation}
This holds, because $(u'-u)(x)=M$ implies that $x\in\Omega$ and then it follows from (\ref{DPPseuraus}) and the fact $u'-u\leq M$ in $\Omega_{\varepsilon}$ that $u'-u=M$ a.e. in $B_\eps(x)$. 

Clearly, property (\ref{cruc1}) (combined with $G\not=\emptyset$) contradicts with the boundedness of set $G\subset\Omega\,$: Let $e_1$ denote the first standard base vector and observe that (\ref{cruc1}) guarantees that if $x\in G$ then 
$G\cap B_{\frac{\eps}{4}}(x+\frac{\eps}{2}e_1)\not=\emptyset\,$. Using this property it is easy to construct a sequence of points $x_k\in G$ such that  $e_1 \cdot x_k$ tends to infinity.
\end{proof}

The following corollary says that the iteration process giving the unique $p$-harmonious function is essentially independent on initial values  $u_0$. 
\begin{corollary}
\label{cor:dpp}
Let
$F:\Gamma_\eps\to \R$ and $u_0:\Omega_\eps\to\R$ be bounded Borel functions so that 
$u_{0_{|\Gamma_{\eps}}}=F$ and let functions $u_j$ be defined by the iteration as   
in Theorem \ref{th:dpp}. Then $u_j$ converges uniformly to the unique $p$-harmonious function $u$ with the boundary values $F$.   
\end{corollary}
\begin{proof}
Suppose  that $|F|\leq C$ in $\Gamma_{\eps}\,$. 
The proof of Theorem \ref{th:dpp} remains valid if one, instead of $u_0=\inf_{\Gamma_{\eps}}F$ in $\Omega$, chooses $u'_0=-C$ or $u''_0=C$ in $\Omega$ as a starting point of the iteration. The only difference is that in the second case the sequence $u''_j$, corresponding the starting point $u_0''$, is decreasing. Since $T$ is order preserving, it follows that $u'_j\leq u_j\leq u''_j$. This implies the claim, since by Theorem \ref{thm:comparison} $u_j'$ and $u_j''$ both converge to the same $p$-harmonious function.
\end{proof}   
The choice of open balls in the definition of the game (and in the DPP) is motivated by the fact that if closed balls are used instead, one faces some nontrivial measurability issues. This is immediately seen by the following example which verifies that the supremum operator  which replaces the function value at $x$ by its supremum in a closed ball around $x$  does not preserve Borel functions.
\begin{example}
\footnote{We are grateful to  Tapani Hyttinen for providing this example to us.}
\label{ex:nonBorel}
There is a bounded Borel function $u:\R^3\to \R$ such that the function $x\mapsto 
\sup_{y\in \overline{B_1}(x)}u(y)$ is not Borel.
\end{example}
\begin{proof}
It is classical (see Theorems 6.7.2 and 6.7.11 in \cite{bogachev06}) that one may pick  a Borel set $A\subset [-1,1]^2$ with the property that its projection to the first coordinate
$$
B=\{ x\in \R\;:\; (x,y)\in A\; {\rm for\; some }\; y\in\R \}
$$
fails to 
be Borel. We identify $\R\subset\R^3$ with the $x_1$-axis, and $\R^2$ with the $(x_1,x_2)$-plane, so that  $A$ is identified as a subset $\R^3$ in a natural manner.
Fix any homeomorphism $\psi:\R^2\to\R^2$ for which the image of the  the
line segment 
satisfies
$$
\psi ([-1,1]\times\{0\})\subset \{ (y_1,y_2)\in\R^2\; :\; y_1^2+y_2^2=1\} .
$$
Finally, define the homeomorphism $\phi:\R^3\to \R^3$
by setting
$$
\phi(x_1,x_2,x_3)=(x_1,\psi(x_2,x_3)).
$$
Then  $\phi(A)$ is a subset of the cylinder
$$
\{(x_1,x_2,x_3)\subset \R^3\; :\; x_2^2+x_3^2=1\}\subset\R^3
$$
and Borel.
Denote by $u:=\chi_{\phi(A)}$ the characteristic function of $\phi(A)$, which is 
a Borel function.  A direct inspection verifies that we have
$$
B=\{ x\in\R^3\; :\;\sup_{y\in \overline{B_1}(x)}u(y)=1\}\cap \R,
$$
and thus $\{ x\in\R^3\; :\;\sup_{y\in \overline{B_1}(x)}u(y)=1\}$ cannot be a Borel set.
It follows that the function $x\mapsto \sup_{y\in \overline{B_1}(x)}u(y)$ fails to be Borel.
\end{proof}

\section{Existence of the value function for the tug of war game}
\label{sec:game}
In this section, we establish connections to the  tug-of-war with noise, see \cite{manfredipr12}. In the tug-of-war with noise a token is placed at $x_0\in \Om$ and a biased coin with probabilities $\alpha$ and $\beta$ ($\alpha+\beta=1$) is tossed. If we get heads (probability $\beta$), then the next point $x_1\in B_\eps(x_0)$ is chosen according to the uniform probability distribution on $B_\eps(x_0)$. On the other hand, if we get tails (probability $\alpha$), then a fair coin is tossed and a winner of the toss moves the token to $x_1\in B_\eps(x_0)$ according to his/her strategy. The players continue playing until they exit the bounded domain $\Om$. Then Player II pays Player I the amount given by the pay-off function $F(x_\tau)$, where $x_\tau$ is the first point outside $\Om$.

\subsection{Existence of measurable strategies}\label{subse:measurability}

In many applications, for example already in the proof of Theorem \ref{th:value} below, one needs to construct strategies  according to the almost infimum or the almost supremum of a given Borel function in a given ball. In order to ensure that this kind of strategies are Borel we need the following lemma.
\begin{lemma}\label{le:st}
Let $u:\Omega_\eps\to\R$ be a bounded Borel function, and let $\delta >0.$ Then there are bounded  Borel functions $S_{\sup},S_{\inf}:\Omega\to\Omega_\eps$ such that $S_{\inf}(x), S_{\sup}(x)\in B_\eps(x)\,$ and
$$
u(S_{\sup} (x))\geq\sup_{B_\eps(x)} u-\delta\quad \qquad 
u(S_{\inf}(x))\leq \inf_{B_\eps(x)} u+\delta\quad
$$
 for all $ x\in\Omega$.
\end{lemma}
\begin{proof}It is enough to find  $S_{\sup}.$ Denote by $\mathcal{B}$ the countable collection of all balls $B\subset \Omega_\eps$ with the rational radius and  with rational coordinates for the centerpoint. For every $B\in\mathcal{B}$ choose a point $x_B\in B$ so that 
$$
v(x_B)\geq \sup_{y\in B} v(y) -\delta/2.
$$ 
Let us denote the collection of all the chosen points by $S:=\{ x_B\; :\;B\in {\mathcal B}\} $, so that $S$ is countable. Since an arbitrary  open ball
$B_r(x)\subset \Omega_\eps$ can be written as a union of balls in $\mathcal{B}$, we obtain for every $x\in\Omega$
$$
\sup_{B_\eps(x)} v\leq \sup_{y\in S\cap B_\eps(x)}v(y)+\delta/2.
$$
The function $S_{\sup}$ is simply obtained by applying Lusin's countable Borel selection theorem (see for example \cite[ p.\ 210]{srivastava98}) to the Borel set
$$
\big\{ (x,y)\in \Omega\times\Omega_\eps\; : \;|x-y|<\eps\;\,{\rm and}\;\; \sup_{B_\eps(x)} v<v(y)-\delta\big \}\cap \big(\R^n\times S\big)\,.\qedhere
$$
\end{proof}

\subsection{Existence of value}\label{subse:value}

We next verify that the unique solution to the DPP obtained in the previous section indeed equals the value function of the game, especially the value of the game is Borel.
Recall that the value of the game for Player I
is given by
\[
u^\eps_\I(x_0)=\sup_{S_{\I}}\inf_{S_{\II}}\,\mathbb{E}_{S_{\I},S_{\II}}^{x_0}[F(x_\tau)]
\]
while the value of the game for Player II is given by
\[
u^\eps_\II(x_0)=\inf_{S_{\II}}\sup_{S_{\I}}\,\mathbb{E}_{S_{\I},S_{\II}}^{x_0}[F(x_\tau)],
\]
where  $x_0$ is a starting point of the game and $S_\I,S_{\II}$ are the strategies employed by the players (more precise definitions are given below). The point $x_\tau$ denotes the first point outside $\Om$, where  $\tau$ refers to the first
time we hit $\Gamma_{\eps}$. The payoff function $F:\Gamma_\eps
\to \R$ is a given, bounded, and Borel measurable. 

A history  of a game up to step $k$ is defined to be a vector of the first $k+1$
game states $x_0,\ldots,x_k$ and $k$ coin tosses $c_1,\ldots,c_k$, that is,
\begin{equation}
\label{eq:history}
\begin{split}
\left(x_0,(c_1,x_1),\ldots,(c_k,x_{k})\right).
\end{split}
\end{equation}
Above $c_j\in \mathcal C:=\{0,1,2\},$ where $0$ denotes that Player I wins, $1$ that Player II wins, and finally $2$ that a random step occurs.

The history of the game is related to the filtration $\{\F_k\}_{k=0}^\infty$, where
${\mathcal F}_0:=\sigma (x_0)$ and
\begin{equation}\label{eq:filtration}
{\mathcal F}_k:=\sigma (x_0, (c_1,x_1),\ldots, (c_k,x_k))\quad \mbox{for}\quad k\geq 1.
\end{equation}

A strategy
$S_\I$
for Player I is a collection of  Borel measurable  functions that give the next
game position given the history of the game. For example
\[
S_\I{\left(x_0,(c_1,x_1),\ldots,(c_k,x_k)\right)}=x_{k+1}\in  B_\eps(x_k)
\]
if Player I wins the toss. Similarly Player II  uses  a strategy $S_{\II}$.

\begin{theorem}\label{th:value}
It holds that $u=u^\eps_\I = u^\eps_\II$, where $u$ is the solution of the DPP
obtained in Theorem \ref{th:dpp}. In particular, the game has the value at each point, and the value function is Borel.
\end{theorem}
\begin{proof}
It is enough to show that 
\begin{equation}\label{eq:II} 
u_\II^\eps \leq u
\end{equation}
because by symmetry $u_\I^\eps \geq u$ and trivially
$u_\II^\eps\geq   u_\I^\eps$.

To prove \eqref{eq:II} we play the game as follows:
Player II follows a strategy
$S_\II^0$ such that at $x_{k-1}\in \Om$ he chooses to step to a
point that almost minimizes $u$, that is, to a point $x_k \in 
B_\eps(x_{k-1})$ such that
\[
u(x_k)\leq\inf_{B_\eps(x_{k-1})} u+\eta 2^{-k}
\]
for some fixed $\eta>0$. This strategy can be chosen to be Borel thanks to Lemma \ref{le:st}.

Then for any strategy $S_\I$ of Player I we  use the  definition of the game  and the fact that $u$ satisfies the DPP to compute:

\[
\begin{split}
\mathbb{E}_{S_\I, S^0_\II}^{x_0}&[u(x_k)+\eta 2^{-k}\,|\F_{k-1}]
\\
&=\frac{\alpha}{2} \left\{ u \big(S^0_\II(x_0,\ldots ,c_{k-1},x_{k-1})\big)+ u\big(S_\I(x_0,\ldots ,c_{k-1},x_{k-1})\big)\right\}\\
&\phantom{aaaaaaaaaaaaa}+ \beta\kint_{ B_\eps(x_{k-1})}
u \ud y+\eta 2^{-k}\\
&\leq \frac{\alpha}{2} \big\{\inf_{ B_\eps(x_{k-1})} u+\eta 2^{-k}+\sup_{  B_\eps(x_{k-1})}
u\big\}+ \beta\kint_{ B_\eps(x_{k-1})}
u \ud y+\eta 2^{-k}\\
&=u(x_{k-1})+\eta 2^{-k}(1+\alpha/2)\\
&\leq u(x_{k-1})+\eta 2^{-(k-1)}.
\end{split}
\]
Above we brutely  estimated the strategy of Player I by $\sup$.

Thus, regardless of the strategy $S_\I$ the process $
M_k=u(x_k)+\eta 2^{-k}
$ is a supermartingale with respect to the filtration $\{ {\mathcal F_k}\}_{k\geq 0}$. Observe that $F(x_\tau)=u(x_\tau)$ and, as the supermartingale is bounded, 
deduce  by optional stopping that
\[
\begin{split}
u_\II^\eps(x_0)&= \inf_{S_{\II}}\sup_{S_{\I}}\,\mathbb{E}_{S_{\I},S_{\II}}^{x_0}[F(x_\tau)]\le \sup_{S_\I} \mathbb{E}_{S_\I,
S^0_\II}^{x_0}[F(x_\tau)+\eta 2^{-\tau}]\\ 
&\leq\sup_{S_\I}  \mathbb{E}^{x_0}_{S_\I, S^0_\II}[u (x_0)+\eta ]=u(x_0)+\eta.
\end{split}
\]
Since $\eta$ was arbitrary this proves the claim.
\end{proof}

One may naturally play the tug-of-war-game, where the strategies for
choosing the next point $x_{k+1}$ are allowed to depend just on part of the information on the game up to time $k.$ As natural examples, we list three different possibilities:
\begin{enumerate}

\item The strategies may depend only on the current point $x_k.$

\item The strategies may depend only on the full history of locations $x_1,\ldots, x_k.$

 \item The strategies may depend only on both the full history of locations $x_1,\ldots, x_k
 $, and on the outcomes of coin tosses until time $k$.

\end{enumerate}
The above proof considers   alternative (3), but actually only uses  strategies from alternative (1). We thus obtain:
\begin{corollary}\label{co:strategies}
The value of the game exists, is Borel, and equals the unique solution of the  DPP, independently of what alternative for the allowed strategies {\rm (1)} -- {\rm (3)} is used. 
\end{corollary}

\section{Continuous modification}
\label{sec:cont}
In this section we present a modification of the game near the boundary which produces continuous value functions for continuous boundary values. Given bounded domain $\Omega\subset\R^n$ as before we assume that the  given  boundary function $F$ is defined in the symmetrically thickened boundary 
$$
\Gamma_{\eps,\eps}:=\{ x\in\R^n\; :\; \dist (x,\partial\Omega) \leq \eps\}.
$$
Let $F:\Gamma_{\eps,\eps}\to\R$ be continuous and bounded. Then we modify the tug-of-war game by playing the game just as before, but at each turn when the game state is at $x\in \Gamma_{\eps,\eps}$, we first flip an extra coin which gives heads with probability $\delta (x)$. If the outcome is  heads, the game is stopped at $x$ and the gain is $F(x)$ as before. If the outcome is tails, we play the usual tug-of-war with noise. We use the choice
\begin{equation}
\label{eq:modification}
\begin{split}
\delta(x):=1-\eps^{-1}\dist (x,\Gamma_\eps)\quad {\rm for}\quad x\in \Gamma_{\eps,\eps}.
\end{split}
\end{equation}
As a result, the game is only modified in $\Gamma_{\eps,\eps}$ -- The unmodified game would correspond to the choice $\delta(x)=\chi_{\Gamma_\eps}(x).$ 

The modified DPP takes the form
\begin{equation}\label{eq:contDPP}
\begin{split}
u(x)=&
(1-\delta(x))\Big(\frac{\alpha}{2}\sup_{B_\eps(x)}u+\frac{\alpha}{2}\inf_{B_\eps(x)}u+\beta\intav_{B_\eps(x)}u\ud y\Big)+\delta(x)F(x)
\end{split}
\end{equation}
for all $x\in \Om_\eps$, with interpretation $\delta(x)=0$ and $\delta(x)F(x)=0$ in $\Om\setminus \Gamma_{\eps,\eps}$. 
Also observe that $\delta(x)=1$ on $\Gamma_\eps$, thus a solution $u$ satisfies $u_{|\Gamma_\eps}=F_{|\Gamma_\eps}.$


\begin{theorem}\label{th:cont} Given a continuous  and bounded boundary function $F:\Gamma_{\eps,\eps}\to \R$, there is a unique continuous  $u:\Omega_\eps\to\R$ that solves the DPP \eqref{eq:contDPP} with the choice  \eqref{eq:modification}.  Moreover, if $F$ is Lipschitz, so is $u$.
\end{theorem}
\begin{proof}
%
Similarly as before, we iterate the operator $T$, where $Tu$ is now given by the right hand side of \eqref{eq:contDPP}.
The main point is that 
\[
\begin{split}
&Tu\in C(\Omega_\eps)\text{ for any bounded }u\in C(\Omega_\eps)\text{ and } \\
&Tu_{|\Gamma_\eps}=F\,.
\end{split}
\]
Hence, if we start the iteration from a constant function  $u_0 \in C(\Om_\eps)$ that is less than $\inf_{x\in\Gamma_{\eps,\eps}} F$, we again
obtain an increasing  and bounded sequence  of continuous functions
$u_k:=T^ku_0$. The uniform convergence of the sequence $u_k$ is proven by an obvious modification of the proof of Theorem \ref{th:dpp}. Hence the limit
$u:=\lim_{k\to\infty} u_k$ is continuous, clearly satisfies
the desired DPP, and the condition $u_{|\Gamma_\eps}=F_{|\Gamma_\eps}.$

The proof of the uniqueness is even simpler: if $u'$ is another solution
and $0<M=u'(x_0)-u(x_0)=\sup_{x\in\Omega}(u'-u),$  where $x_0\in\Omega,$ then necessarily $u'-u\equiv M$ in full $\eps$-neighbourhood of $x$. A finite iteration leads to a contradiction at the boundary layer $\Gamma_\eps.$

Next we prove Lipschitz continuity of $u$ if  $F$ is $L_0$-Lipschitz and $\|F\|_\infty\leq M_0.$  It  is enough to show that there is $L=L(L_0,M_0, \alpha,\beta,\eps)\geq L_0$ such that  for arbitrary   $u\in C(\Omega_\eps)$ satisfying
\begin{equation}\label{eq:Lip0}
\|u\|_\infty\leq M \quad {\rm and }\quad u\;\, {\rm is}\;\, L{\rm-Lipschitz},
\end{equation}
the function $Tu$ also satisfies this conditions.

Assume \eqref{eq:Lip0}.  Obviously $\| Tu\|_\infty\leq M$.
Since $Tu|_{\Gamma_\eps}=F_{|\Gamma_\eps}$,  it is clearly enough to verify that 
\begin{equation}\label{eq:Lip1}
\Lip_{Tu}(x)\leq L\quad {\rm for\; any}\;\; x\in \Omega.\qquad 
\end{equation}
Here the pointwise Lipschitz constant is defined for any function $g$ by
$$
\Lip_{g}(x):=\limsup_{y\to x}\frac{|g(y)-g(x)|}{|y-x|}.
$$
Indeed, in order to then show that $Tu$ is $L$-Lipschitz, pick arbitrary $x,x'\in\Omega_\eps$. If the open interval $(x,x')$ is contained in $\Omega$,
the needed estimate follows from \eqref{eq:Lip1}. If not, 
let $y$
be the closest point to $x$ in $[x,x']\cap \Gamma_\eps$ and define $y'$ in an analogous manner. Apply \eqref{eq:Lip1} on intervals $[x,y]$
and $[x',y']$, and use the $L_0$-Lipschitz property of $F$ on
$[y,y']$.

The assumption \eqref{eq:Lip0} also implies for functions $x\mapsto \inf_{B_\eps(x)}u$ and $x\mapsto \sup_{B_\eps(x)}u$, $x\in \Omega,$ that
\begin{equation}\label{eq:Lip2}
\Lip_{\inf_{B_\eps(x)}u}(x),\Lip_{\sup_{B_\eps(x)}u}(x)\leq L\;\; {\rm and}\;\; 
|\inf_{B_\eps(x)}u|,|\sup_{B_\eps(x)}u|\leq M_0.
\end{equation}
Below we deduce for $x\mapsto \intav_{B_\eps(x)}u\ud y$, that
\begin{equation}\label{eq:Lip3}
\Lip_{\intav_{B_\eps(x)}u \ud y}\leq 2n  \eps^{-1}M_0\;\; {\rm and}\;\; 
\abs{\intav_{B_\eps(x)}u\ud y}\leq M_0,\qquad x\in \Omega.
\end{equation}
The upper bound is clear. The Lipschitz estimate is a consequence of the simple geometric estimate for the symmetric difference 
\begin{equation*}
\begin{split}
\frac{\abs{B_\eps(x)\Delta B_\eps(z)}}{|B_\eps(x)|}&\leq \frac{2n\abs{B_{\eps+|x-z|}(x)\setminus B_\eps(x)}}{  |B_\eps(x)|}\\
&=2n  |x-z|\eps^{-1}+O(|x-z|^2).
\end{split}
\end{equation*}
Consequently
\[
\begin{split}
\abs{\intav_{B_\eps(x)}u\ud y-\intav_{B_\eps(z)}u\ud y}&= \abs{\int_{B_\eps(x)}u\ud y-\int_{B_\eps(z)}u\ud y}/\abs{ B_\eps(x)}\\
& \le M_0 (2n  |x-z|\eps^{-1}+O(|x-z|^2)),
\end{split}
\]
which implies the estimate in \eqref{eq:Lip3}.

We will the apply the  above bounds and the well known facts
$$
\Lip_{fg}\leq |f|\Lip_g+|g|\Lip_f\quad {\rm and}\quad \Lip_{c_1f+c_2g}\leq |c_1|\Lip_f+|c_2|\Lip_g.
$$
Denote $h(x):=\frac{\alpha}{2}\sup_{B_\eps(x)}u+\frac{\alpha}{2}\inf_{B_\eps(x)}u+\beta\intav_{B_\eps(x)}u.$ Then  
$$
\Lip_h(x)\leq \frac{\alpha}{2} L+\frac{\alpha}{2}L+\beta 2n  \eps^{-1}M_0,\quad {\rm and}\quad |h(x)|\leq M_0\quad x\in\Omega.
$$
In addition, directly by definition $\Lip_\delta\leq \eps^{-1}$,  and thus  $\Lip_{\delta F}(x)\leq L_0+M_0\eps^{-1}$.
Hence for any $x\in \Omega$
\begin{equation*}
\begin{split}
L_{Tu}(x)&\leq  \eps^{-1}M_0+1\cdot (\alpha L+\beta 2n  \eps^{-1}M_0)+
(L_0+M\eps^{-1})\\ 
&=C(L_0,M_0,\alpha,\eps)+\alpha L\leq L
\end{split}
\end{equation*} 
if $L$ satisfies $L\geq (1-\alpha)^{-1}C(L_0,M_0,\alpha,\eps).$
This finishes the proof.
\end{proof}

One should observe that the bound for $L$ in the above proof blows up (as one expects) as $\eps\to 0.$ On the other hand, one would expect that under suitable conditions  a uniform H\"older bound should remain true.

After the previous result, the proof of the following theorem is completely analoguous to that of Theorem \ref{th:value}.
\begin{theorem}\label{th:contvalue}
The continuous modification of the game has value, and the value function is the unique, continuous solution of the DPP {\rm \eqref{eq:contDPP}} given by Theorem {\rm \ref{th:cont}}. This statement is independent of the alternative ways  {\rm  (1)--(3)} at the end of Section \ref{sec:game} of taking into account the history of the game.
\end{theorem}

\section{Parabolic}

The existence, uniqueness and comparison principle for solutions to the parabolic dynamic programming principle, see Definition \ref{defi:p-parabolious} below, are considerably easier to obtain than the elliptic counterparts. These functions are also value functions for the tug-of-war with noise with a limited number of rounds, and possibly varying boundary values. In \cite{manfredipr10c} they are called $(p,\eps)$-parabolic because of their connection to solutions of normalized $p$-parabolic equation.

Recall that $\Om\subset \Rn$ is a bounded open set. Then $\Om_\infty=\Om\times (0,\infty)$ denotes a space-time cylinder. To prescribe boundary values, we
denote the boundary strip by
\[
\begin{split}
\Gamma_p^\eps= \Big((\Om_\eps \setminus \Om )\times (-\frac{\eps^2}{2},\infty) \Big)\cup
\Big(\Om\times (-\frac{\eps^2}{2},0] \Big).
\end{split}
\]
Boundary values are given by $F: \Gamma_p^\eps\to \R$ which is a bounded Borel function. 
\begin{definition}
\label{defi:p-parabolious}
Let $\alpha+\beta=1,\ \alpha,\beta\ge0$. The function $u$ satisfies the parabolic dynamic programming principle with
boundary values $F$ if
\[
\begin{split}
u(x,t)&= \frac{\alpha}{2}
\left\{ \sup_{y\in  B_\eps (x)} u\Big(y,t-\frac{\eps^2}{2}\Big) +
\inf_{y\in  B_\eps (x)} u\Big(y,t-\frac{\eps^2}{2}\Big) \right\}
\\
&\hspace{1 em}+\beta \kint_{B_\eps (x)} u\Big(y,t-\frac{\eps^2}{2}\Big)
\ud y 
\end{split}
\]
for every $(x, t) \in \Omega_\infty$, and 
\[
\begin{split}
u(x,t) &= F(x,t), \qquad \textrm{for every}\quad (x,t) \in
\Gamma_p^\eps.
\end{split}
\]
\end{definition}

Denote by $\mathcal F_\eps$  the space of bounded functions defined on $\Gamma_p^\eps\cup \Omega_\infty$ that are Borel measurable on each time slice $\{(x,s):x\in \Om_\eps,s=t\}$. We define the operator $T$ on $\mathcal F_\eps$ analogously to what we did in the elliptic case:
\[
\begin{split}
Tu(x,t)=\begin{cases}
&\frac{\alpha}{2}
\left\{\sup_{y\in  B_\eps (x)} u\Big(y,t-\frac{\eps^2}{2}\Big) +
\inf_{y\in  B_\eps (x)} u\Big(y,t-\frac{\eps^2}{2}\Big) \right\}
\\
&\hspace{1 em}+\beta \kint_{B_\eps (x)} u\Big(y,t-\frac{\eps^2}{2}\Big)
\ud y \qquad \mbox{for every } (x, t) \in \Omega_\infty
\\
&u(x,t), \qquad \textrm{for every}\quad (x,t) \in
\Gamma_p^\eps.
\end{cases}
\end{split}
\]

\begin{theorem}
Given  initial and boundary values $F$, there exists a unique solution to the parabolic DPP.
\end{theorem}

\begin{proof}
Let $u_0 \in \mathcal F_\eps$ be an arbitrary function with boundary values $F$. First observe that $T$ preserves the class $\mathcal F_\eps$, see \eqref{eq:sup-meas}.
We claim that the iteration by $u_{i+1}=T(u_i)$, $i=0,2,\ldots$ converges in a finite number of steps at any fixed $(x,t)\in \Om_\infty$. To establish this, we use induction, and show that when calculating the values $$u_{i+1}=Tu_i$$ only the values for $t>i\eps^2/2$ can change. This is clear if $i=0$ since then the operator $T$ uses the values from $[-\eps^2/2,0]$ which are given by $F$. Suppose then that this holds for some $k$ i.e.\ only the values $t>k\eps^2/2$ of $u_k$ can change in the operation $Tu_k$. Then from
\[
\begin{split}
Tu_{k+1}= T(Tu_k)
\end{split}
\]
we deduce that $T$ can only change the values for $t>(k+1)\eps^2/2 $ for $u_{k+1}$. Clearly, the limit satisfies Definition~\ref{defi:p-parabolious}. Moreover, the uniqueness follows by the same induction argument.   
\end{proof}

Suppose that $F_u\ge F_v$, and  $u_0, v_0 \in \mathcal F_\eps$ are arbitrary functions with boundary values $F_u$ and $F_v$ respectively. Then $Tu_0\ge Tv_0$ for $t\leq \eps^2/2$ and by iterating this argument similarly as above, we obtain a comparison principle.
\begin{theorem}
Suppose that $u$ and $v$ satisfy the parabolic DPP with the boundary values $F_u\ge F_v$. Then $u\ge v$.
\end{theorem}

\def\cprime{$'$} \def\cprime{$'$}


\end{document}